\documentclass[12pt]{article}
\usepackage{amsmath,amssymb,mathrsfs,wasysym}
\usepackage{amsthm}
\usepackage{enumerate}
\usepackage{graphicx}
\date{}
\newtheorem{theorem}{Theorem}
\newtheorem{prop}{Proposition}
\newtheorem{lemma}{Lemma}
\newtheorem{rem}{Remark}
\newtheorem{exmp}{Example}
\newtheorem{cor}{Corollary}

\begin{document}
\author{\bf Edyta Bartnicka}
  \title{Stars of graphs of projective codes}
\maketitle

\begin{abstract}
Let $\Gamma_k(V)$ be the Grassmann graph whose vertex set  is formed by all $k$-dimensional subspaces of an $n$-dimensional vector space $V$ over the finite field $F_q$ consisting of $q$ elements. We discuss its subgraph $\Pi(n,k)_q$  formed by projective  codes. We show that there are precisely two types of maximal cliques in $\Pi(n,k)_q$: stars and tops.  We give a complete description of stars, i.e., maximal cliques consisting of all $k$-dimensional projective codes containing a certain  $(k-1)$-dimensional subspace of $V$.
\end{abstract}


\vspace{0.3cm}


\section{Introduction}
Grassmann graphs, which have already been discussed  in a number of books and papers (see, e.g.
\cite{Pankov2010, Pankov2015}), are important for many reasons, for instance, they are classical examples of distance regular graphs (for more details
see \cite{brouwer}).

The Grassmann graph  $\Gamma_k(V)$ is the simple graph whose vertex set ${\mathcal G}_{k}(V)$ is formed by $k$-dimensional subspaces of an $n$-dimensional vector space $V$ over the $q$-element field $F_q$. It can be identified with the graph of all linear  codes $[n, k]_q$. Two distinct codes are connected by an edge if they have the maximal possible number of common codewords, i.e. their intersection is $(k-1)$-dimensional. We only deal with the case  $1<k<n-1$ since the Grassmann graph is complete for $k=1$ and $k=n-1$.

This paper concerns the restriction of the Grassmann graph to the set $\Pi(n,k)_q$ of all projective codes $[n,k]_{q}$, i.e., linear codes such that columns in their generator matrices are non-zero and mutually non-proportional. They form the subgraph of the Grassmann graph denoted by $\Pi[n, k]_q$.  We refer to \cite{KP3} for the description of some properties of graphs of projective codes.
The subgraphs of the Grassmann graph formed by various types of linear
codes are considered, e.g. in 
\cite{CGK, CG, KP2, Pankov2023}.
Maximal cliques in these graphs  are investigated in \cite{BM, KP4} for example.

We direct our attention to maximal cliques  of $\Pi[n,k]_q$. They are intersections of maximal cliques of $\Gamma_k(V)$ with the set $\Pi(n,k)_q$. There are exactly two types of maximal cliques in $\Gamma_k(V)$: stars $[S\rangle_k$ (made up by all $k$-dimensional subspaces of $V$ containing a fixed $(k-1)$-dimensional subspace $S$) or tops (consisting of all $k$-dimensional subspaces of $V$ contained in a fixed $(k+1)$-dimensional subspace). Any non-empty intersection of a star or a top with the set $\Pi(n, k)_q$ is a clique, but it need not be a maximal clique. If this intersection is a maximal clique of $\Pi(n,k)_q$, then we call it a star or a top  of $\Pi[n,k]_q$, respectively.  Here we  provide the thorough description of stars  of $\Pi[n, k]_q$. 
We consider two cases of $S$ separately. First we  assume that a generator matrix for a $(k-1)$-dimensional code $S$ does not contain zero columns (Section \ref{non}), and secondly, we study the opposite case (Section \ref{deg}). The starting point of both investigations  is showing that if a generator matrix for  $S$ does not contain proportional columns, then the intersection of a star $[S\rangle_k$ with the set $\Pi(n, k)_q$  is a star of $\Pi[n, k]_q$. Next we study intersections of  stars $[S\rangle_k$ with the set $\Pi(n, k)_q$ such that generator matrices for $S$ contain proportional columns.   We specify the number of projective codes in such intersections. 
Finally, we indicate all parameters for generator matrix for $S$ to $[S\rangle_k\cap\Pi(n, k)_q$ be a maximal clique of $\Pi(n, k)_q$ (Theorems \ref{notprojectiveS} and \ref{degenerateS}). Also, we show that there are maximal cliques of $\Pi(n, k)_q$ which are stars and tops of $\Pi[n,k]_q$ simultaneously.
 Examples of maximal cliques of $\Pi(n, k)_q$
 for different cases of $S$
 are also given.

\section{Graph of projective  codes} 
Let $F_q$ be the finite field consisting of $q$ elements  and let  $V=F_{q}^{n}$ be the  $n$-dimensional vector space  over $F_q$. Consider  the Grassmannian  ${\mathcal G}_{k}(V)$ consisting of all $k$-dimensional subspaces of $V$.
The {\it Grassmann graph} $\Gamma_{k}(V)$ is the simple and connected graph whose vertex set is ${\mathcal G}_{k}(V)$. Its two vertices comprise an edge if, and
only if, their intersection is $(k-1)$-dimensional.
The Grassmann graph is complete if $k=1$ or $k=n-1$, and that is why we suppose that $1<k<n-1$ throughout the paper.


One of the basic ideas of graph theory is that of a clique. A clique in
an undirected graph is a subset of the set of its vertices such that any two distinct
elements from this subset are joined by an edge. 
If a clique is not
properly contained in any clique, then it  is said to be {\it maximal}.
Any  maximal clique of $\Gamma_{k}(V)$ is one of the following types:
the {\it star}  $[S\rangle_{k}$ consisting of all $k$-dimensional subspaces of $V$ containing a fixed subspace $S\in {\mathcal G}_{k-1}(V)$, 
or the {\it top}  $\langle U]_{k}$ consisting of all $k$-dimensional subspaces of $V$ contained in a fixed  subspace $U\in {\mathcal G}_{k+1}(V)$.

Recall that $V$ contains precisely
$\genfrac{[}{]}{0pt}{}{n}{1}_q=\genfrac{[}{]}{0pt}{}{n}{n-1}_q=
[n]_{q}=\frac{q^n-1}{q-1}$ $1$-dimensional subspaces 
 and the total number of elements of $\langle S]_{k}$ is equal $[n-k+1]_{q}$.
The intersection of two distinct maximal cliques of the  form $\langle S]_{k}$  contains at most one vertex.
The intersection of $[S\rangle_{k}$ and $\langle U]_{k}$ is either empty or it contains precisely $q+1$ elements. The latter is realized exactly if
the corresponding $(k-1)$-dimensional and $(k+1)$-dimensional subspaces $S, U$ are incident. 


Any element of ${\mathcal G}_{k}(V)$ is interpreted as a {\it linear code} $[n,k]_{q}$ and it is called {\it non-degenerate} if 
its generator matrix does not contain zero columns.
If 
columns in a generator matrix for a linear code  are non-zero and mutually non-proportional, then this code 
is said to be {\it projective}.

Write $\Pi[n, k]_q$  for the restriction of the  Grassmann graph to the set $\Pi(n, k)_q$ of all projective codes. This is the simple  graph  whose set of vertices is $\Pi(n,k)_{q}$. 
It was shown in \cite[Theorem 1]{KP3} that in the case when $q\geq \binom{n}{2}$, the graph $\Pi[n, k]_q$ is connected.
Any clique of $\Pi(n,k)_{q}$ is a  clique of $\Gamma_{k}(V)$ since $\Pi[n, k]_q$  is a subgraph of $\Gamma_{k}(V)$. And so any maximal clique of $\Pi(n,k)_{q}$ is
 the intersection of $\Pi(n,k)_{q}$ with a maximal clique of $\Gamma_{k}(V)$.  We denote intersections $[S\rangle_{k}\cap{\Pi}(n,k)_q$ and $\langle U]_{k}\cap{\Pi}(n,k)_q$ by $[S\rangle^{\Pi}_{k}$ and $\langle U]^{\Pi}_{k}$, respectively.
They both
can be empty or  non-maximal cliques of $\Pi(n,k)_{q}$, otherwise they are called {\it a star} or {\it a top} of $\Pi[n, k]_q$, respectively.



\section{Basic results}

Consider a $(k-1)$-dimensional linear code $S$. The set $[S\rangle^{\Pi}_{k}$ is formed by all $k$-dimensional projective codes $Q_i$ containing $S$. A generator matrix $M_i$ for any  code $Q_i$  can be obtained from a generator matrix $M=\left[\begin{matrix}
                                                                                                                    v_1\\
                                                                                                                    v_2\\
                                                                                                                    \vdots\\
                                                                                                                    v_{k-1}
                                                                                                                  \end{matrix}\right]$ for $S$ by adding a certain $n$-dimensional row vector $w_i$ which is not a linear combination of the rows from $M$. 
                                                                                                                  We can assume without loss of generality that  $M=\left[I_{k-1}\ \  A\right]$,
where $I_{k-1}$ means the identity ($k-1$)-matrix, $A$ is $(k-1)\times (n-k+1)$ matrix with $n\geqslant k+2$ (since $1<k< n-1$). Then we can also suppose that the first $k-1$ coordinates of $w_i$ are $0$. The  set of all such vectors $w_i$ will be denoted by $W$.

Assume that two proportional columns in $M$  are equivalent and write $l(k_j)$ for the number of columns in the equivalence class $[k_j]$ of a column $k_j$ (the $j-th$ column of $M$), $l(S)$ for the number of all equivalence classes. The number of zero columns of $M$ will be denoted by $c(S)$.

The above notations we shall use throughout the paper.
\begin{rem}\label{lkj}
  It is obvious that if there exists a column $k_j$ of $M$ such that $l(k_j)>q$ or if $c(S)>1$, then $[S\rangle^{\Pi}_{k}$ is an empty set (there are proportional or even zero columns in  any matrix obtained from $M$ by adding an $n$-dimensional row vector).
\end{rem}
\begin{prop}
  \label{dimw1}
   If $S$ is  a $(k-1)$-dimensional  subspace of $V$ and $\dim\langle W\rangle=1$, there is no any $(k+1)$-dimensional subspace $U$ of $V$ such that $[S\rangle^{\Pi}_{k}=\langle U]^{\Pi}_{k}$.
\end{prop}
\begin{proof}
 Assume that $\dim\langle W\rangle=1$. This is the case, when there exists precisely one vector $w_1$ such that $=\left[\begin{matrix}
                                                                                                                    M\\
                                                                                                                    w_1
 \end{matrix}\right]$ is a projective code, or equivalently, $q=2$ and  $A$ consists of:
  \begin{enumerate}
   \item mutually non-proportional columns $k_{j'_1},k_{j'_2}, \dots , k_{j'_t}$, with $3\leqslant t\leqslant k-1$ (this is the case of non-degenerate linear code $S$), or

 \item  a zero column and  mutually non-proportional columns $k_{j'_1},k_{j'_2}, \dots , k_{j'_t}$, where $2\leqslant t\leqslant k-1$,
  \end{enumerate}
and  
$k_{j'_1},k_{j'_2}, \dots , k_{j'_t}$ are  exactly the same as  columns $k_{j_1},k_{j_2}, \dots, k_{j_t}$ of $I_{k-1}$, respectively. So,  for any $p\in \{1, 2, \dots t\}$ the $j_p$-th and the $j'_p$-th coordinates of the $j_p$-th row of $M$ are its only non-zero elements. Any other row of $M$ has exactly one non-zero coordinate (on the main diagonal of $I_{k-1}$) and the last $n-k+1$ coordinates of $w_1$ are non-zero.

 Any $(k+1)$-dimensional subspace $U$ of $V$ such that $[S\rangle^{\Pi}_{k}\subseteq\langle U]^{\Pi}_{k}$ is a projective code whose generator matrix is of the form $\left[\begin{matrix}
                                                                                                                    M\\
                                                                                                                    w_1\\
                                                                                                                    u
 \end{matrix}\right]$ and we can suppose that the first $k-1$ coordinates of a row vector $u$ are $0$.

 Let $S$ be non-degenerate. Consider a matrix obtained from $\left[\begin{matrix}
                                                                                                                    M\\
                                                                                                                    w_1
 \end{matrix}\right]$  by adding the vector $u$ to one from the row vectors $v_i$ of $M$.  Such matrix generates a $k$-dimensional projective code which is not an element of $[S\rangle^{\Pi}_{k}$, and so $[S\rangle^{\Pi}_{k}$ is a proper subset of $\langle U]^{\Pi}_{k}$.
 Similarly, we can indicate  an element of the set $\langle U]^{\Pi}_{k}$ which does not belong to $[S\rangle^{\Pi}_{k}$ if there is a zero column in generator matrices for $S$. Namely, such code is generated by  a matrix obtained from $\left[\begin{matrix}
                                                                                                                    M\\
                                                                                                                    w_1
 \end{matrix}\right]$  by adding the vector $u$ to any two row vectors $v_i$ of $M$ if a coordinate of $u$ in the zero column is different from all the remaining among the last $n-k+1$ coordinates of $u$. Otherwise, that is when a coordinate of $u$ in the zero column is the same as the one in columns  $k_{j'_{p_1}},k_{j'_{p_2}}, \dots , k_{j'_{p_{l}}}$, where $1\leqslant l\leqslant t-1$, then the searched matrix is obtained from  $\left[\begin{matrix}
                                                                                                                    M\\
                                                                                                                    w_1
 \end{matrix}\right]$ by adding $u$ to the $k_{j_p}$-th row, where $k_{j'_p}$ is one of the columns $k_{j'_{p_1}},k_{j'_{p_2}}, \dots , k_{j'_{p_{l}}}$.
\end{proof}
The above proposition implies the following:
\begin{cor}
  \label{dimw1nie}
  If $S$ is  a $(k-1)$-dimensional  subspace of $V$ and $\dim\langle W\rangle=1$, then  $[S\rangle^{\Pi}_{k}$ is not a maximal clique  of $\Pi(n, k)_q$.
\end{cor}
\begin{lemma}
  \label{dimwi}
  Let $S$ be  a $(k-1)$-dimensional  subspace of $V$. 
  If $\dim\langle W\rangle>2$, then  $[S\rangle^{\Pi}_{k}$ is a star of $\Pi[n, k]_q$.
\end{lemma}
\begin{proof}
 Assume that $\dim\langle W\rangle>2$. It implies $\left|[S\rangle^{\Pi}_{k}\right|>2$. Hence $[S\rangle^{\Pi}_{k}$ is not properly  contained in any  star of $\Pi[n, k]_q$ since there is at most one element in the intersection of two distinct stars of ${\mathcal G}_{k}(V)$.

$[S\rangle^{\Pi}_{k}$ is also not contained in any  clique $\langle U]^{\Pi}_{k}$ of $\Pi(n,k)_{q}$, otherwise\linebreak
$\dim U\geqslant \dim\langle \{v_1, \dots, v_{k-1},W\}\rangle=\dim S+\dim \langle W\rangle$, and therefore $\dim W\leqslant k+1-(k-1)=2$, what contradicts the assumption. This finishes the proof.
\end{proof}
Of course, when $\left|[S\rangle^{\Pi}_{k}\right|> q+1$, the set $[S\rangle^{\Pi}_{k}$ is a maximal clique of $\Pi(n, k)_q$. As it was said before, such $[S\rangle^{\Pi}_{k}$ is not contained in any other star of $\Pi[n, k]_q$, but also in any top of $\Pi[n, k]_q$. It follows from the fact that
 the intersection of a star and a top of ${\mathcal G}_{k}(V)$ contains no more than $q+1$ elements.
\begin{rem}\label{dimw2}
 Let $S$ be  a $(k-1)$-dimensional  subspace of $V$ and $\dim\langle W\rangle=2$. Notice  that  $U=\left\langle \{v_1, \dots, v_{k-1}, w_1, w_2\}\right\rangle$ is the only one  $(k+1)$-dimensional subspace of $V$ such that $\langle U]^{\Pi}_{k}$ contains $[S\rangle^{\Pi}_{k}$.
\end{rem}
\begin{prop}
  \label{startop}
Let $S$ be  a $(k-1)$-dimensional  subspace of $V$. The set $[S\rangle^{\Pi}_{k}$ is  a star of $\Pi[n, k]_q$ and  $\dim\langle W\rangle=2$ if, and only if, $[S\rangle^{\Pi}_{k}=\langle U]^{\Pi}_{k}$, where $U$ is a $(k+1)$-dimensional subspace of $V$ such that  $U=\left\langle \{v_1, \dots, v_{k-1}, w_1, w_2\}\right\rangle$.
\end{prop}
\begin{proof} Let us suppose first that $\dim\langle W\rangle=2$. On account of Remark \ref{dimw2} $[S\rangle^{\Pi}_{k}=\langle U]^{\Pi}_{k}$ if $[S\rangle^{\Pi}_{k}$ is  a maximal clique of $\Pi(n, k)_q$ and $[S\rangle^{\Pi}_{k}\subsetneq\langle U]^{\Pi}_{k}$ if $[S\rangle^{\Pi}_{k}$ is not a maximal clique of $\Pi(n, k)_q$.

According to the proof of Lemma \ref{dimwi} we get immediately that $[S\rangle^{\Pi}_{k}\neq \langle U]^{\Pi}_{k}$ for any  $(k+1)$-dimensional subspace $U$ of $V$ if $\dim\langle W\rangle>2$. Proposition \ref{dimw1} now yields to desired claim.
\end{proof}
\begin{lemma}\label{permutation}
A clique $[S\rangle^{\Pi}_{k}$ is a star of $\Pi[n,k]_{q}$ if, and only if, $[S_\sigma\rangle^{\Pi}_{k}$ is a star of $\Pi[n,k]_{q}$ for any linear code $S_\sigma$  whose generator matrix is obtained from $M=\left[I_{k-1}\ \  A\right]$ by a permutation of the set of columns of $A$.
\end{lemma}
\begin{proof}
Let $l_i$,  where $i=1, 2, \cdots, {n-k+1}$,  be the $i$-th column of $A$, that is $M=\left[I_{k-1}\ \  A\right]=\left[I_{k-1}\ l_1\ l_2\ \dots\ l_{n-k+1}\right],$ and let     $\sigma$ be a permutation of the set $\{1, 2, \dots, n-k+1\}$. Then a generator matrix for $S_\sigma$ is of the form $M_\sigma=\left[I_{k-1}\ \  A_\sigma\right]=\left[I_{k-1}\ l_{\sigma(1)}\ l_{\sigma(2)}\ \dots\  l_{\sigma(n-k+1)}\right].$
Obviously, a $n$-dimensional vector $w_i=[0, \dots, 0, w_i^1, \dots, w_i^{n-k+1}]\in W$ if, and only if, $\left[\begin{matrix}
                                                                                                                    M_\sigma\\
                                                                                                                    w_{i}^\sigma
\end{matrix}\right]$ is a generator matrix for $S_\sigma$, where $w_i^\sigma=[0, \dots, 0, w_i^{\sigma(1)}, \dots, w_i^{\sigma(n-k+1)}].$ Thus $\dim \langle W\rangle=\dim \langle W_\sigma\rangle$, where $W_\sigma=\{w_i^\sigma; i=1, 2, \cdots, {n-k+1}\}$. In consequence, the proof is straightforward from Corollary \ref{dimw1nie} and Lemma \ref{dimwi} if $\dim \langle W\rangle\neq2$.

Suppose now that $\dim \langle W\rangle=2$. In view of Proposition \ref{startop} $[S\rangle^{\Pi}_{k}$ is a star of $\Pi[n,k]_{q}$ exactly if a  generator matrix for any projective code from $\langle U]^{\Pi}_{k}$, where $U=\left\langle v_1, \dots, v_{k-1}, w_1, w_2\right\rangle$, is of the form $\left[\begin{matrix}
                                                                                                                    M\\
                                                                                                                    w_{i}
\end{matrix}\right]$. Equivalently,  for any permutation $\sigma$ of the set $\{1, 2, \dots, n-k+1\}$ a matrix $\left[\begin{matrix}
                                                                                                                    M_\sigma\\
                                                                                                                    w_{i}^\sigma
\end{matrix}\right]$ generates a $k$-dimensional projective code. Such projective codes are the only elements of   $\langle U_\sigma]^{\Pi}_{k}$, where a $(k+1)$-dimensional subspace $U_\sigma$ of $V$ is generated by the row vectors of $M_\sigma$, $w_{1}^\sigma$ and $w_{2}^\sigma$.
\end{proof}
\section{Stars $\boldsymbol{[S\rangle^{\Pi}_{k}}$ defined by a non-degenerate $\boldsymbol{S}$}
\label{non}

In  this section, we assume that $S$ is non-degenerate and then we investigate stars $[S\rangle^{\Pi}_{k}$ which are defined by such $S$.
Obviously, if $S$ is a projective code, then  $[S\rangle^{\Pi}_{k}=\langle S]_{k}$ is a star of $\Pi[n, k]_q$. Consider the case when $S$ is not a projective code. Then  $[S\rangle^{\Pi}_{k}$ is a proper subset of $\langle S]_{k}$ and there are at least two proportional columns in $M$.

\begin{lemma}\label{cardinality}
  If $S$ is  a $(k-1)$-dimensional  non-degenerate  linear code which does not belong to $\Pi(n, k-1)_q$, then
  $$\left|[S\rangle^{\Pi}_{k}\right|=\frac{\left[\prod_{k_j\in L}(q-1)(q-2)\cdots(q-l(k_j)+1)\right]q^{l(S)-k+1}}{q-1},$$ where $L$ denotes the set of
equivalence classes of the  columns of $M$ such that $l(k_j)>1$.
\end{lemma}
\begin{proof}
 There are precisely $q(q-1)\cdots(q-l(k_j)+1)$ possibilities of choosing coordinates of $w_i$ corresponding to columns from $[k_j]$. For any projective code of the set $[S\rangle^{\Pi}_{k}$ its generator matrix $\left[\begin{array}{c}
M\\
w_i
\end{array}\right]$ can be formed by  $q^{k-1}$ mutually non-proportional vectors $w_i$. Consequently there are exactly
$$\frac{\prod_{k_j\in L'}q(q-1)(q-2)\cdots(q-l(k_j)+1)}{(q-1)q^{k-1}}$$  projective codes containing $S$, where $L'$ denotes the set of all  equivalence classes of the columns of $M$. Simple transformations lead to the desired equation.
\end{proof}

\begin{theorem}\label{notprojectiveS}
Let $S$ be  a $(k-1)$-dimensional  non-degenerate  linear code which does not belong to $\Pi(n, k-1)_q$. The set $[S\rangle^{\Pi}_{k}$ is a star of $\Pi[n, k]_q$ if, and only if,  $l(k_j)\leqslant q$ for all columns $k_j$ of $M$  and one of the following cases occurs:
\begin{enumerate}
  \item $q=2$ and \begin{itemize}
                    \item $l(S)=k=3,\ n=6$;
                    \item $l(S)\geqslant k+1$;
                  \end{itemize}
  \item $q=3$ and \begin{itemize}
  \item $l(S)=k-1,\ |L|\geqslant3$;
                    \item $l(S)=k,\ |L|\geqslant2$;
                    \item $l(S)\geqslant k+1$;
                  \end{itemize}
  \item $q=4$ and \begin{itemize}
  \item $l(S)=k-1=1$;
  \item $l(S)=k-1,\ |L|\geqslant2$;
                    \item $l(S)\geqslant k$;
                  \end{itemize}
  \item $q\geqslant5$,
\end{enumerate} where $L$ denotes the set
of equivalence classes of the  columns of $M$ such that $l(k_j)>1$.
\end{theorem}

\begin{proof}
Whenever we will give the number of elements of $[S\rangle^{\Pi}_{k}$ in this proof, we always do it according to  Lemma \ref{cardinality} without references.

We start off by showing that the set $[S\rangle^{\Pi}_{k}$ is a star of $\Pi[n, k]_q$ in any of the mentioned cases.

$"\Leftarrow"$\ Let $l(k_j)\leqslant q$ for any column $k_j$ of $M$ and
\begin{enumerate}
  \item $q=2$.

  Suppose that $l(S)=k=3,\ n=6$. There are precisely two elements in the set  $\left[S\right\rangle^{\Pi}_{k}$. Taking into account Lemma \ref{permutation} we can assume without loss of generality that $M=\left[\begin{matrix}
                                                                                                                    v_1\\
                                                                                                                    v_2
                                                                                                                  \end{matrix}\right]=\left[\begin{matrix}
                                                                                                                    1 & 0 & 1 & 1 & 0 & 1\\
                                                                                                                    0 & 1 & 1 & 0 & 1 & 1
                                                                                                                  \end{matrix}\right]$. Then $w_1=[0,0,0,1,1,1],  w_2=[0,0,1,1,1,0]$
  and by Remark \ref{dimw2} $U=\langle\{v_1, v_2, w_1, w_2\}\rangle$ is the only one $4$-dimensional subspace of $V$ such that $\langle U]^{\Pi}_{k}$ contains $[S\rangle^{\Pi}_{k}$.

  \noindent
  Suppose that $[S\rangle^{\Pi}_{k}$ is a non-maximal clique of $\Pi(n,k)_{q}$. In consequence, there exists a $3$-dimensional projective code $Q$ of $\langle U]^{\Pi}_{k}$ such that $\dim (Q\cap Q_1)=\dim (Q\cap Q_2)=2$. Hence $Q=\langle\{v', w'_1, w'_2\}\rangle$, where $v'\in S, w'_1\in Q_1$, $w'_2\in Q_2$ and $w'_1,w'_2\notin S$. Then $Q$ contains a vector $\alpha_1v_1+\alpha_2v_2+w_1+w_2$ such that for any $\alpha_1, \alpha_2\in F_q$  either four coordinates are $0$ or all of them are $1$. Obviously, a $3$-dimensional subspace of $U$ containing such  vector is not projective.
  It means that $[S\rangle^{\Pi}_{k}=\langle U]^{\Pi}_{k}$ and in the light of Proposition \ref{startop} this is a maximal clique of $\Pi(n,k)_{q}$.

  If $l(S)\geqslant k+1$, then $\left|[S\rangle^{\Pi}_{k}\right|\geqslant q^2>q+1$. So $[S\rangle^{\Pi}_{k}$ is a star of $\Pi[n, k]_q$.
  \item $q=3$.

Consider three cases: $l(S)=k-1$ and  $|L|\geqslant4$
                        or $l(S)=k$ and $|L|\geqslant2$
                      or $l(S)\geqslant k+1$.
                    Then $\left|[S\rangle^{\Pi}_{k}\right|$ is greater than or equal to $(q-1)^3, (q-1)q, q^2$, respectively.  Thus $\left|[S\rangle^{\Pi}_{k}\right|>q+1$ and $[S\rangle^{\Pi}_{k}$ is a maximal clique of $\Pi(n,k)_{q}$.

The proof for $l(S)=k-1$ and  $|L|=3$ is more complicated since $\left|[S\rangle^{\Pi}_{k}\right|=(q-1)^2=4=q+1$.
We use the form  $M=\left[I_{k-1}\ \ A\right]$. Notice that $k-1\geqslant3$  and for each of three certain columns $k_{j_1}$, $k_{j_2}$,  $k_{j_3}$ of $I_{k-1}$ a matrix $A$ contains one or two proportional columns and there are no any other columns in $A$. 
Consider the case of $A$ with two proportional columns to each of the columns $k_{j_1}$, $k_{j_2}$,  $k_{j_3}$ of $I_{k-1}$. By Lemma \ref{permutation} we can assume without loss of generality that the $j_1-th$, the $j_2-th$, the $j_3-th$ rows of $A$ make up the matrix:
$$\left[\begin{matrix}x_1& 0& 0& x'_1& 0&0\\
0& x_2& 0& 0& x'_2&0\\
0& 0& x_3& 0& 0&x'_3\end{matrix}\right],
$$
where $x_1, x_2, x_3,x'_1, x'_2, x'_3
\in F_q\backslash\{0\}$. The remaining entries of $A$ are zero.
Then we can write  the non-zero
part of the vectors $w_1, \dots, w_4$ in the following way:
$[x_1, x_2, x_3, -x'_1, -x'_2,-x'_3
]$, $[x_1, x_2, -x_3, -x'_1, -x'_2,x'_3
]$, $[x_1, -x_2, x_3, -x'_1, x'_2,-x'_3
]$, $[-x_1, x_2, x_3, x'_1, -x'_2,-x'_3
]$, respectively. 
It is easy to check that $w_1=w_2+w_3+w_4$ and the set $\{w_2,w_3,w_4\}$ is linearly independent and this holds independently from the last three coordinates of the vectors $w_i$. 
In the light of Lemma \ref{dimwi} $[S\rangle^{\Pi}_{k}$ is  maximal.

  \item $q=4$.

  If $l(S)=k-1,\ |L|\geqslant2$ or
                   $l(S)\geqslant k$,  $\left|[S\rangle^{\Pi}_{k}\right|$ is greater than or equal to $(q-1)(q-2), (q-2)q$, respectively.  So, in both cases $\left|[S\rangle^{\Pi}_{k}\right|>q+1$ and thereby $[S\rangle^{\Pi}_{k}$ is a star of $\Pi[n,k]_{q}$.

  In the case of $l(S)=k-1=1$ we get immediately that $|L|=1$,  $[S\rangle^{\Pi}_{k}$ contains precisely two elements and $M=\left[\begin{matrix}
                                                                                                                    v
                                                                                                                  \end{matrix}\right]=\left[\begin{matrix}
                                                                                                                    1 & x_1 & x_2 & x_3
                                                                                                                  \end{matrix}\right]$, where $x_1, x_2 ,x_3\in \{1, x, x+1\}$ are elements of $F_4\backslash\{0\}$. And so we can write that $w_1=[0,x_1, xx_2, (x+1)x_3],$  $w_2=[0,x_1, (x+1)x_2, xx_3].$
  By Remark \ref{dimw2} the subspace $U=\langle\{v, w_1, w_2\}\rangle$ is the only one $3$-dimensional subspace of $V$ such that $\langle U]^{\Pi}_{k}$ contains $[S\rangle^{\Pi}_{k}$.
  Assume that $[S\rangle^{\Pi}_{k}$ is a non-maximal clique of $\Pi(n,k)_{q}$. In consequence there exists a $2$-dimensional projective code $Q$ of $\langle U]^{\Pi}_{k}$ such that  $Q=\langle\{w'_1, w'_2\}\rangle$, where $w'_1\in Q_1$, $w'_2\in Q_2$ and $w'_1,w'_2\notin S$. Then $Q$ contains a vector  \begin{align*} & w=\alpha_1w_1+\alpha_2w_2+\alpha_3v=[\alpha_3, \left((\alpha_1+\alpha_2)+\alpha_3\right)x_1,\\
   &\left((x\alpha_1+(x+1)\alpha_2)+\alpha_3\right)x_2,\left(((x+1)\alpha_1+x\alpha_2)+\alpha_3\right)x_3] \end{align*} with $\alpha_1, \alpha_2\in F_4\backslash\{0\}, \alpha_3\in F_4$.
According to the operations in $F_4$ 
 one of the scalars $\alpha_1+\alpha_2,\ x\alpha_1+(x+1)\alpha_2,\ (x+1)\alpha_1+x\alpha_2$ is 0 and the other two are the same element of $F_4$. Thus  one of the scalars $\alpha_1+\alpha_2+\alpha_3,\ x\alpha_1+(x+1)\alpha_2+\alpha_3,\ (x+1)\alpha_1+x\alpha_2+\alpha_3$ is equal $\alpha_3$  and  the other two are the same element of $F_4$. Any matrix of the form $\left[\begin{matrix} a&a&b&b\\
                                c&d&e&f\end{matrix}\right]$ over $F_4$ contains two proportional columns. Therefore any
 $2$-dimensional subspace of $U$ containing $w$ is not projective, and so we get a contradiction.
  It means that $[S\rangle^{\Pi}_{k}=\langle U]^{\Pi}_{k}$  and in view of Proposition \ref{startop} this is a maximal clique of $\Pi(n,k)_{q}$.


  \item $q\geqslant5$.

  Suppose first that $l(S)=k-1$, $q=5$ and $|L|\geqslant2$. We get immediately that $\left|[S\rangle^{\Pi}_{k}\right|\geqslant(q-1)(q-2)>q+1$.

  \noindent
  If $l(S)=k-1$ and $q\geqslant6$, then $\left|[S\rangle^{\Pi}_{k}\right|\geqslant(q-2)(q-3)$ and thereby $\left|[S\rangle^{\Pi}_{k}\right|>q+1$  for all $q\geqslant6$.

  \noindent
  Similarly, $l(S)\geqslant k$ yields  $\left|[S\rangle^{\Pi}_{k}\right|\geqslant(q-2)q>q+1$ for all $q\geqslant5$.

  In the last case, namely $q=5,\ l(S)=k-1$ and $|L|=1$ we obtain $\left|[S\rangle^{\Pi}_{k}\right|=(q-2)(q-3)=6=q+1$. A matrix $A$ consists of three or four columns proportional to the $j$-th column of $I_{k-1}$. Thus the only non-zero elements of $A$ are  three or four coordinates $x_1, x_2, x_3, x_4\in F_q\backslash\{0\}$ in  the $j$-th row, respectively. Consider $A$ with four columns. Then we can write the non-zero part of the vectors $w_i$ as:
   $[x_1, 2x_2, 3x_3, 4x_4
]$, $[x_1, 2x_2, 4x_3, 3x_4
]$, $[x_1, 3x_2, 2x_3, 4x_4
]$, $[x_1, 3x_2, 4x_3, 2x_4
]$, $[x_1, 4x_2, 2x_3, 3x_4
]$ and $[x_1, 4x_2, 3x_3, 2x_4
]$.  
It is easy to see that $\dim\langle\{w_i,\ i=1, \dots, 6\}\rangle=3$ and this same holds in the case of $A$ with three columns.  In the light of Lemma \ref{dimwi} $[S\rangle^{\Pi}_{k}$ is a star of $\Pi[n,k]_{q}$.
\end{enumerate}

$"\Rightarrow"$ We will show that  $[S\rangle^{\Pi}_{k}$ is not a maximal clique of $\Pi(n,k)_{q}$ in any other case than the above. According to Remark \ref{lkj} it holds if $l(k_j)> q$ for some column $k_j$ of $M$. Suppose then that $l(k_j)\leqslant q$ for any column $k_j$ of $M$. In most of the studied cases we present a generator matrix $X$ for a $k$-dimensional projective code $Q$ such that $\dim\left(Q\cap Q_i\right)=k-1$ for any $Q_i\in [S\rangle^{\Pi}_{k}$, 
which is equivalent to not be $[S\rangle^{\Pi}_{k}$  a star of $\Pi[n,k]_{q}$.
Consider all the possibilities of $q$ again:
\begin{enumerate}
  \item $q=2$;

  We get immediately that $\left|[S\rangle^{\Pi}_{k}\right|=q^0=1$ if $l(S)=k-1$. On account of Corollary \ref{dimw1nie} $[S\rangle^{\Pi}_{k}$ is not a maximal clique of $\Pi(n,k)_{q}$.

  Assume now that $l(S)=k=3$ and $n=5$. Then there are exactly two codes in $[S\rangle^{\Pi}_{k}$.
  In view of Lemma \ref{permutation} it suffices to consider  three cases of $M$. 
  We present them with the corresponding vectors $w_1, w_2$ and generator matrices $X$ below. 
\begin{itemize}
              \item $M=\left[\begin{matrix}v_1\\v_2\end{matrix}\right]=\left[\begin{matrix}1&0&1&1&0\\0&1&1&0&1\end{matrix}\right]$, $w_1=[0,0,0,1,1]$, $w_2=[0,0,1,1,1]$, $X=\left[\begin{matrix}v_1+v_2\\w_2\\v_1+w_1\end{matrix}\right]=\left[\begin{matrix}1&1&0&1&1\\0&0&1&1&1\\1&0&1&0&1\end{matrix}\right]$
              \item $M=\left[\begin{matrix}v_1\\v_2\end{matrix}\right]=\left[\begin{matrix}1&0&1&1&1\\0&1&1&0&1\end{matrix}\right]$, $w_1=[0,0,0,1,1]$, $w_2=[0,0,1,1,0]$, $X=\left[\begin{matrix}v_1\\v_2+w_1\\v_2+w_2\end{matrix}\right]=\left[\begin{matrix}1&0&1&1&1\\0&1&1&1&0\\0&1&0&1&1\end{matrix}\right]$
              \item $M=\left[\begin{matrix}v_1\\v_2\end{matrix}\right]=\left[\begin{matrix}1&0&1&0&1\\0&1&1&1&1\end{matrix}\right]$, $w_1=[0,0,0,1,1]$, $w_2=[0,0,1,1,0]$, $X=\left[\begin{matrix}v_2\\v_1+w_1\\v_1+w_2\end{matrix}\right]=\left[\begin{matrix}0&1&1&1&1\\1&0&1&1&0\\1&0&0&1&1\end{matrix}\right]$
            \end{itemize}
If $l(S)=k\geqslant4$, then also $\left|[S\rangle^{\Pi}_{k}\right|=q=2$. 
Furthermore, there exists one or two proportional columns $k_i, k_j$  in $M$, which are not proportional to any column of $I_{k-1}$. Then the $i$-th coordinate of $w_1,\ w_2$ it is $0, 1$, respectively and similarly the $j$-th coordinate of $w_1,\ w_2$ it is $1, 0$, respectively. Thus $k_i, k_j$ are not proportional in $$X=\left[\begin{matrix}v_1+v_2\\v_2+v_3\\ \vdots\\ v_{k-2}+v_{k-1}\\w_1\\w_2\end{matrix}\right].$$

For any  column $k_{i'}$  of $I_{k-1}$ there exists at most one proportional column $k_{j'}$  in $A$. But then the $i'$-th  coordinate of $w_1,\ w_2$ it is $0$, and the $j'$-th coordinate of $w_1,\ w_2$ it is $1$. It means that $k_{i'}, k_{j'}$ are not proportional in $X$. It must be pointed out that columns of $M$ are  proportional if, and only if, they are proportional in $\left[\begin{matrix}v_1+v_2\\v_2+v_3\\ \vdots\\ v_{k-2}+v_{k-1}\end{matrix}\right]$, and therefore there are no proportional columns in $X$. Therefrom it follows that $Q$ is projective.
  \item $q=3$;

  The form of $X$ in the rest of the proof (also for $q=4$) will be as follows: $X=\left[\begin{matrix}v_1\\\vdots\\v_{i-1}\\ v_i+w_1\\v_{i+1}\\ \vdots\\ v_{k-1}\\w_2\end{matrix}\right].$

  Let us assume first that $l(S)=k-1,\ |L|=2$, then $\left|[S\rangle^{\Pi}_{k}\right|=q-1=2$. A matrix $A$ consists of two columns $k_{i_1}, k_{i_2}$ proportional to a column $k_i$ of $I_{k-1}$ and one or two columns $k_{j_1}, k_{j_2}$ proportional to a column $k_j$ of $I_{k-1}$, where $i\neq j$. So, there are three or four columns of $A$.
  Consider  $A$ with four columns. By Lemma \ref{permutation} we can assume without loss of generality that the $i-th$ and the $j-th$ rows of $A$ make up the matrix:
$$\left[\begin{matrix}
x_1&x_2&0&0\\
0&0&x_3&x_4\end{matrix}\right],$$
  where  $x_1, x_2, x_3, x_4\in F_q\backslash\{0\}$. In consequence, we can write the non-zero part of $w_1, w_2$ as
$[x_1, -x_2, x_3, -x_4]$, $[x_1, -x_2, -x_3, x_4],$
  respectively.

  A matrix $X'$ obtained from entries $X$ in columns $k_i, k_j, k_{i_1}, k_{i_2}, k_{j_1}, k_{j_2}$ that are in the $i$-th,  the $j$-th rows and in $w_2$, respectively, is of the form 
   $$X'=\left[\begin{matrix}1&0&-x_1&0&x_3&-x_4\\0&1&0&0&x_3&x_4\\ 0&0&x_1&-x_2&-x_3&x_4\end{matrix}\right].$$ 
As we can see, columns of $X'$ are mutually non-proportional and so there are no any proportional columns in $X$.

Let now $l(S)=k, \ |L|=1$. It follows then that $\left|[S\rangle^{\Pi}_{k}\right|=q=3$. A matrix $A$ consists of three columns and there are precisely two possibilities. One of them is $A$ with all columns mutually proportional. Then they are not proportional to any column of $I_{k-1}$, and so each of them has at least two non-zero elements. Suppose that there are in the $i$-th and in the $j$-th rows two such elements of columns $A$. According to Lemma \ref{permutation} we can assume without loss of generality that the $i-th$ and the $j-th$ rows of $A$ make up the matrix:
$$\left[\begin{matrix}x_1&x_1&x_2\\\alpha x_1&\alpha x_1&\alpha x_2\end{matrix}\right],$$
where $x_1, x_2, \alpha\in F_q\backslash \{0\}$.
Then we can write the non-zero part of $w_1, w_2,w_3$ as
 $[0, x_1, 2x_2]$, $[2x_1, 0, x_2]$, $[2x_1, x_1, 0],$
  respectively. 
  As we can see $w_3=w_1+w_2$ and hence the intersection $Q$ with $Q_1$, $Q_2$ or $Q_3$ is $(k-1)$-dimensional.
  The coordinates of the last three columns of $X$ in the $i$-th, the $j$-th rows and in $w_2$ are the following:
  $$\left[\begin{matrix}x_1\\\alpha x_1\\ 2x_1\end{matrix}\right], \left[\begin{matrix}2x_1\\\alpha x_1\\ 0\end{matrix}\right], \left[\begin{matrix}0\\\alpha x_2\\ x_2\end{matrix}\right],$$ respectively. Thus there are no any proportional columns in $X$.

  The second possibility is $A$ consisting of two columns $k, k'$ proportional to the $i$-th column of $I_{k-1}$, and one  column $k''$ non-proportional to any other in $M$. There are at least two non-zero elements in $k''$. So, we can suppose that one of them is in the $j$-th row, where $j\neq i$. Therefore and by Lemma \ref{permutation} we can assume without loss of generality that the $i-th$ and the $j-th$ rows of $A$ make up the matrix:
$$\left[\begin{matrix}x_1&\alpha x_1&x_2\\0&0&x_3\end{matrix}\right],$$
where $x_1, x_3, \alpha\in F_q\backslash \{0\}$, $x_2\in F_q$.
Then we can write the non-zero part of $w_1, w_2,w_3$ as
 $[x_1, 2\alpha x_1, x_1]$, $[x_1, 2\alpha x_1, 2x_1]$, $[x_1, 2\alpha x_1, 0],$
  respectively. 
  It is easy to see that $w_3=2(w_1+w_2)$ and hence the intersection $Q$ with $Q_1$, $Q_2$ or $Q_3$ is $(k-1)$-dimensional.
  The coordinates of the last three columns of $X$ in the $i$-th, the $j$-th rows and in $w_2$ are the following:
  $$\left[\begin{matrix}2x_1\\0\\x_1\end{matrix}\right], \left[\begin{matrix}0\\0\\2\alpha x_1\end{matrix}\right], \left[\begin{matrix}x_2+x_1\\ x_3\\ 2x_1\end{matrix}\right],$$ respectively. This shows that $Q$ is projective.
  \item $q=4$.

  We will consider the case when $l(S)=k-1\geqslant 2$ and  $|L|=1$. Then $\left|[S\rangle^{\Pi}_{k}\right|=q-2=2$.  A matrix $A$ consists of three columns  proportional to the $j$-th column of $I_{k-1}$ and at least two rows, but all rows different from the $j$-th one are zero. The $j$-th  row of $A$ is of the form $[x_1, x_2, x_3]$, where $x_1, x_2 ,x_3\in \{1, x, x+1\}$ are elements of $F_4\backslash\{0\}$. So we can write
the non-zero part of the vectors  $w_1, w_2$ 
 in the following way: $[x_1, xx_2, (x+1)x_3]$, $[x_1, (x+1)x_2, xx_3],$
  respectively. 
  The coordinates of the last three columns of $X$ in the $i$-th, the $j$-th rows and in $w_2$, where $j\neq i$, are as follows:
  $$\left[\begin{matrix}x_1\\ x_1\\ x_1\end{matrix}\right], \left[\begin{matrix}xx_2\\x_2\\ (x+1)x_2\end{matrix}\right], \left[\begin{matrix}(x+1)x_3\\ x_3\\ xx_3\end{matrix}\right].$$ Therefore there are no any proportional columns in $X$.
\end{enumerate}
\end{proof}

\begin{exmp}
We show two cliques $[S\rangle^{\Pi}_{k}$ with two elements in  the case of $q=4$, $l(S)=k-1$ and $|L|=1$.
\begin{enumerate}[a)]
  \item Let $M=\left[\begin{matrix}1&1&1&1\end{matrix}\right]$ be a generator matrix for $S\in {\mathcal G}_{1}(F_4^4)$. 
      Generator matrices for $Q_1, Q_2\in [S\rangle^{\Pi}_{2}$ are as follows:
  $$\left[\begin{matrix}1&1&1&1\\0&1&x&x+1\end{matrix}\right],\ \ \ \left[\begin{matrix}1&1&1&1\\0&1&x+1&x\end{matrix}\right].$$
  The only one $3-$dimensional subspace of $F_4^4$ containing $S, Q_1$ and $Q_2$ is $U=\langle\{[1,1,1,1],[0,1,x,x+1],[0,1,x+1,x]\}\rangle.$
  There are precisely four elements of $\left\langle U\right]_{2}$ do not containing $S$. We can write generator matrices for them as:
  $\left[\begin{matrix}0&1&x&x+1\\0&1&x+1&x\end{matrix}\right],$ $\left[\begin{matrix}0&1&x&x+1\\1&0&x&x+1\end{matrix}\right], \left[\begin{matrix}1&0&x+1&x\\0&1&x+1&x\end{matrix}\right],$ $\left[\begin{matrix}1&0&x+1&x\\1&0&x&x+1\end{matrix}\right].$ All they are not projective. Therefrom it follows that  $[S\rangle^{\Pi}_{2}= \left\langle U\right]^{\Pi}_{2}$ is a star and a top of $\Pi[4, 2]_4$.

  \item Choose now the set $[S\rangle^{\Pi}_{3}$ such that $$\left[\begin{matrix}v_1\\v_2\\w_1\end{matrix}\right]=\left[\begin{matrix}1&0&1&1&1\\0&1&0&0&0\\0&0&1&x&x+1\end{matrix}\right],\ \ \ \left[\begin{matrix}v_1\\v_2\\w_2\end{matrix}\right]=\left[\begin{matrix}1&0&1&1&1\\0&1&0&0&0\\0&0&1&x+1&x\end{matrix}\right]$$ are generator matrices for its elements $Q_1, Q_2$. Consider the code $Q$ whose generator matrix is $\left[\begin{matrix}v_1\\v_2+w_1\\w_2\end{matrix}\right]=\left[\begin{matrix}1&0&1&1&1\\0&1&1&x&x+1\\0&0&1&x+1&x\end{matrix}\right].$
      This is a $3$-dimensional projective code which satisfies the equalities\linebreak $\dim (Q\cap Q_1)=\dim (Q\cap Q_2)=2,$ and so the set $\{Q, Q_1, Q_2\}$ is a clique of $\Pi(5, 3)_4$. Consequently $[S\rangle^{\Pi}_{3}$ is not a maximal clique of $\Pi(5, 3)_4$.
\end{enumerate}

\end{exmp}
\section{Stars  $\boldsymbol{[S\rangle^{\Pi}_{k}}$ defined by a degenerate $\boldsymbol{S}$}
\label{deg}

Throughout this section we will be concerned with cliques $[S\rangle^{\Pi}_{k}$ defined by $(k-1)$-dimensional degenerate linear code $S$ and we assume that there is precisely one zero column of $M$. We will write $k_z$ for such a column and $M'$  for the $(k-1)\times (n-1)$ matrix obtained from $M$ by removing the zero column. Let $V'$ be a hyperplane of $V$. Then a linear code with generator matrix  $M'$ is an element of the Grassmannian  ${\mathcal G}_{k-1}(V')$. This code will be denoted by $S'$.

 A generator matrix for any $k$-dimensional projective code containing $S'$ can be obtained from $M'$ by adding a certain $(n-1)$-dimensional row vector $w_i'$ which is not a linear
combination of the rows from $M'$. We can also assume that the first $k-1$ coordinates of $w_i'$ are zero.
\begin{theorem}\label{S'rzut}
If $S'$ is a $(k-1)$-dimensional projective code, then $[S\rangle^{\Pi}_{k}$ is a star of $\Pi[n, k]_q$.
\end{theorem}
\begin{proof}
Suppose that  $S'$ is a $(k-1)$-dimensional projective code, i.e. $M'$ does not contain proportional columns. Generator matrices for $k$-dimensional projective codes containing $S$ are of the form $\left[\begin{matrix}M\\ w_i\end{matrix}\right]$, where is met one of the following conditions:
\begin{itemize}
  \item the coordinate of $w_i$ in $k_z$ is any non-zero element of $F_q$ and removing this coordinate from $w_i$ gives $w_i'$.
  The number of such codes is equal to the number of $k$-dimensional projective codes containing $S'$ multiplied by $q-1$. In the case of  $k<n-2$ all such  codes containing $S'$ make up the star $[S'\rangle^{\Pi}_{k}=[S'\rangle_{k}$ (since $S'$ is projective) of $\Gamma_k(V')$. Consequently $|[S'\rangle^{\Pi}_{k}|=[(n-1)-k+1]_q=[n-k]_q>[n-(n-2)]_q=[2]_q$.
  If $k=n-2$,  the number of $k$-dimensional projective codes containing $S'$ is equal to the number of $1$-dimensional subspaces of $F_q\times F_q$, i.e., $[2]_q$.
  \item the only non-zero coordinate of $w_i$  is the one in $k_z$. Clearly, all such matrices generate the same projective code.
\end{itemize}
Thus we get 
$|[S\rangle^{\Pi}_{k}|\geqslant[2]_q(q-1)+1=q^2>q+1$ for all $q\geqslant2$. Therefore $[S\rangle^{\Pi}_{k}$ is a maximal clique of $\Pi(n, k)_q$.
\end{proof}

 \begin{lemma}\label{cardinality2}
  If $M'$ contains  proportional columns, then
  $$\left|[S\rangle^{\Pi}_{k}\right|=\left[\prod_{k_j\in L}(q-1)(q-2)\cdots(q-l(k_j)+1)\right]q^{l(S')-k+1},$$ where $L$ denotes the set of
equivalence classes of the  columns of $M$ such that $l(k_j)>1$.
\end{lemma}
\begin{proof}
 The only $k$-dimensional projective codes containing $S$ are generated by $\left[\begin{matrix}M\\ w_i\end{matrix}\right]$, where the coordinate of $w_i$ in $k_z$ is any non-zero element of $F_q$ and removing the zero column from $w_i$ gives $w_i'$. Hence the cardinality of $[S\rangle^{\Pi}_{k}$ is equal to the number of $k$-dimensional projective codes containing $S'$ multiplied by $q-1$. By the same method as in the proof of Lemma \ref{cardinality} we can show that the number of such codes containing $S'$ is $$\frac{\left[\prod_{k_j\in L}(q-1)(q-2)\cdots(q-l(k_j)+1)\right]q^{l(S')-k+1}}{(q-1)},$$ which yields the desired equation.
\end{proof}
 \begin{theorem}\label{degenerateS}
Let $M'$ contains proportional columns.  The set $[S\rangle^{\Pi}_{k}$ is a star of $\Pi[n, k]_q$ if, and only if,  $l(k_j)\leqslant q$ for all columns $k_j$ of $M$ and one of the following cases occurs:
\begin{enumerate}
  \item $q=2$ and \begin{itemize}
                    \item $l(S')=k=3,\ k<n-2$;
                    \item $l(S')\geqslant k+1,\ k<n-2$;
                  \end{itemize}
  \item $q=3$ and \begin{itemize}
  \item $l(S')=k-1=1$; \item $l(S')=k-1$, $|L|\geqslant2$;
                    \item $l(S')\geqslant k$;
                  \end{itemize}
  \item $q\geqslant4$,
\end{enumerate} where $L$ denotes the set
of equivalence classes of the  columns of $M$ such that $l(k_j)>1$.
\end{theorem}
\begin{proof}
We will use Lemma \ref{cardinality2} to give the number of elements of $[S\rangle^{\Pi}_{k}$ in this proof and we suppose that $M'$ contains proportional columns.

$"\Leftarrow"$\ Let $l(k_j)\leqslant q$ for any column $k_j$ of $M$ and
\begin{enumerate}
  \item $q=2$.

  If $l(S')=k=3,\ k<n-2$, then  $\left|[S\rangle^{\Pi}_{k}\right|=q=2$. Taking into account Lemma \ref{permutation}
  it suffices to consider four cases of $M$. We show them with the corresponding vectors $w_1, w_2$:  

  $\left[\begin{matrix}
  1&0&1&1&0&1&0\\
  0&1&1&0&1&1&0
  \end{matrix}\right]$, \ $\left[\begin{matrix}0,0,0,1,1,1,1\end{matrix}\right]$, $\left[\begin{matrix}0,0,1,1,1,0,1\end{matrix}\right]$,

  $\left[\begin{matrix}
  1&0&1&1&0&0\\
  0&1&1&0&1&0
  \end{matrix}\right]$, \ $\left[\begin{matrix}0,0,0,1,1,1,\end{matrix}\right]$, $\left[\begin{matrix}0,0,1,1,1,1\end{matrix}\right]$,

  $\left[\begin{matrix}
  1&0&1&1&1&0\\
  0&1&1&0&1&0
  \end{matrix}\right]$, \ $\left[\begin{matrix}0,0,0,1,1,1\end{matrix}\right]$, $\left[\begin{matrix}0,0,1,1,1,1\end{matrix}\right]$,

$\left[\begin{matrix}
  1&0&1&0&1&0\\
  0&1&1&1&1&0
  \end{matrix}\right]$, \ $\left[\begin{matrix}0,0,0,1,1,1\end{matrix}\right]$, $\left[\begin{matrix}0,0,1,1,0,1\end{matrix}\right]$.

\noindent
The rest of this proof is almost identical with the one from\linebreak
 Theorem \ref{notprojectiveS} ($"\Leftarrow",\ q=2,\ l(S)=k=3,\ n=6$).
 Assuming that $[S\rangle^{\Pi}_{k}$ can be extended by including $Q$ and showing that $Q$ contains a vector $\alpha_1v_1+\alpha_2v_2+w_1+w_2$ such that for any $\alpha_1, \alpha_2\in F_q$  either four coordinates are $0$ or at least five of them are $1$ leads to a contradiction and we thus get that $[S\rangle^{\Pi}_{k}$ is a star and a top of $\Pi[n,k]_{q}$  simultaneously.


  Suppose now that $l(S')\geqslant k+1$. Then $\left|[S\rangle^{\Pi}_{k}\right|\geqslant q^2=4>3=q+1$, and so $[S\rangle^{\Pi}_{k}$ is a star of $\Pi[n, k]_q$.
  \item $q=3$.

  We assume first that $l(S')=k-1=1$. It is clear that $n=4,\ |L|=1$ and   $\left|[S\rangle^{\Pi}_{k}\right|=q-1=2$. A generator matrix for $S$ is of the form $M=\left[\begin{matrix}v\end{matrix}\right]=\left[\begin{matrix}1&0&x_1&x_2\end{matrix}\right],$ where $x_1, x_2\in F_3\backslash\{0\}$, and then $w_1=\left[0,1,x_1,2x_2\right],$ $w_2=\left[0,2,x_1,2x_2\right]$.
  Also here we assume that $[S\rangle^{\Pi}_{k}$ can be extended by including $Q$ and in the same manner as before we can get a contradiction. It follows from the fact that
  $Q$ contains a vector $\alpha_1v+\alpha_2w_1+\alpha_3w_2$, where $\alpha_1\in F_q,\ \alpha_2, \alpha_3\in F_q\backslash\{0\}$, whose coordinates satisfying one of the three conditions: at least two  are $0$ or at least three  are the same or these are two pairs of the same elements.
And so $Q$ is not projective. Thereby $[S\rangle^{\Pi}_{k}$ is a star and a top of $\Pi[n, k]_q$.

  If $l(S')=k-1$, $|L|\geqslant3$ or $l(S')\geqslant k$, then $\left|[S\rangle^{\Pi}_{k}\right|\geqslant(q-1)^3=8>q+1$ or $\left|[S\rangle^{\Pi}_{k}\right|\geqslant(q-1)q=6>q+1$, respectively. Therefore $[S\rangle^{\Pi}_{k}$ is a maximal clique of $\Pi(n,k)_{q}$ in both cases.

  Let now $l(S')=k-1,\ |L|=2$. Consequently $\left|[S\rangle^{\Pi}_{k}\right|\geqslant(q-1)^2=4=q+1$.
 A matrix $A$ consists of zero column $k_z$ and one or two proportional columns for each of two certain columns $k_{j_1}$, $k_{j_2}$ of $I_{k-1}$. 
  Consider the case of $A$ with five columns. In view of Lemma \ref{permutation} we can assume without loss of generality that  the $j_1-th$ and the $j_2-th$ rows of $A$ make up the matrix
$\left[\begin{matrix}0&x_1&0&x'_1&0\\0&0&x_2&0&x'_2\end{matrix}\right],$
where $x_1, x_2, x'_1, x'_2
\in F_q\backslash\{0\}$. All the remaining elements of  $A$ are zero.
Then we can  write the
non-zero part of the vectors $w_1, \dots, w_4$ as:
$ [x_1,x_1, x_2, -x'_1, -x'_2]$,
$[x_1,x_1, -x_2, -x'_1, x'_2]$,
$[-x_1,x_1, x_2, -x'_1, -x'_2]$,
$[-x_1, x_1,-x_2, -x'_1, x'_2],$ respectively. 
It is easy to check that $w_4=-w_1+w_2+w_3$ and the set $\{w_1,w_2,w_3\}$ is always linearly independent,
regardless from the last two coordinates of $w_1, \dots, w_4$.
In the light of Lemma \ref{dimwi} $[S\rangle^{\Pi}_{k}$ is a maximal clique of $\Pi(n,k)_{q}$.

 \item $q\geqslant4$.

In this case $\left|[S\rangle^{\Pi}_{k}\right|\geqslant(q-1)(q-2)$ and so $\left|[S\rangle^{\Pi}_{k}\right|>q+1$ for any $q\geqslant4$, what finishes the proof.
  \end{enumerate}

  $"\Rightarrow"$ We will prove that  $[S\rangle^{\Pi}_{k}$ is not a maximal clique of $\Pi(n,k)_{q}$ in any other case than the above. This is straightforward from Remark \ref{lkj} if $l(k_j)> q$ for some $k_j$ of $M$. That is why we suppose that $l(k_j)\leqslant q$ for any column $k_j$ of $M$.

Consider the case $q=2$ first.

  There is exactly one element in $[S\rangle^{\Pi}_{k}$ if $l(S')=k-1$ and according to Corollary \ref{dimw1nie} $[S\rangle^{\Pi}_{k}$ is not maximal.

  The proof for $l(S')=k\geqslant4$ is exactly the same as for non-degenerate $S$ in Theorem \ref{notprojectiveS}.

  Let $l(S')=k=n-2=3$. Then $\left|[S\rangle^{\Pi}_{k}\right|=q$.
  By Lemma \ref{permutation} it suffices to consider  three cases of $M$. 
  We present them with the corresponding vectors $w_1, w_2$ and generator matrices $X$  for  $k$-dimensional projective codes $Q$ such that $[S\rangle^{\Pi}_{k}\cup Q$ is a clique of $\Pi(n,k)_{q}$.
  These are: 
 \begin{enumerate}
              \item $M=\left[\begin{matrix}v_1\\v_2\end{matrix}\right]=\left[\begin{matrix}1&0&0&1&1\\0&1&0&1&1\end{matrix}\right]$, $w_1=[0,0,1,0,1]$, $w_2=[0,0,1,1,0]$, $X=\left[\begin{matrix}v_1+w_2\\v_2+w_1\\v_1+v_2\end{matrix}\right]=\left[\begin{matrix}1&0&1&0&1\\0&1&1&1&0\\1&1&0&0&0\end{matrix}\right];$
              \item $M=\left[\begin{matrix}v_1\\v_2\end{matrix}\right]=\left[\begin{matrix}1&0&0&1&1\\0&1&0&1&0\end{matrix}\right]$, $w_1=[0,0,1,0,1]$, $w_2=[0,0,1,1,1]$, $X=\left[\begin{matrix}v_1+w_1\\v_2\\w_2\end{matrix}\right]=\left[\begin{matrix}1&0&1&1&0\\0&1&0&1&0\\0&0&1&1&1\end{matrix}\right];$
              \item $M=\left[\begin{matrix}v_1\\v_2\end{matrix}\right]=\left[\begin{matrix}1&0&0&1&0\\0&1&0&1&1\end{matrix}\right]$, $w_1=[0,0,1,0,1]$, $w_2=[0,0,1,1,1]$, $X=\left[\begin{matrix}v_1\\v_2+w_1\\w_2\end{matrix}\right]=\left[\begin{matrix}1&0&0&1&0\\0&1&1&1&0\\0&0&1&1&1\end{matrix}\right].$
            \end{enumerate}

 Suppose now that  $q=3$, $l(S')=k-1\geqslant2$ and $|L|=1$. Then $\left|[S\rangle^{\Pi}_{k}\right|=q-1=2$ and there are precisely three columns of $A$, namely $k_z$ and two columns $k_{j_1},k_{j_2}$ proportional to a column $k_j$ of $I_{k-1}$. On account of Lemma \ref{permutation} we can assume without loss of generality that the $j-th$ and the $i-th$ rows, where $i\neq j$, make up the matrix
$\left[\begin{matrix}0&x_1&x_2\\0&0&0\end{matrix}\right],
$
where $x_1, x_2$ run through all the elements of $F_q\backslash\{0\}$. So we can  write the
non-zero part of the vectors $w_1, w_2$ as
$[1,x_1,2x_2]$, $[2, x_1,2x_2]$,
respectively.  It is easy to see that $X=\left[\begin{matrix}v_1\\  \vdots\\v_{i-1}\\ v_i+w_1\\v_{i+1}\\ \vdots\\ v_{k-1}\\w_2\end{matrix}\right]$ generates $Q$ such that $[S\rangle^{\Pi}_{k}\cup Q$ is a clique of $\Pi(n,k)_{q}$. Thus  $[S\rangle^{\Pi}_{k}$ is not a star of $\Pi[n,k]_{q}$.
  \end{proof}
\begin{exmp}\begin{enumerate}[a)]
  \item Consider $[S\rangle^{\Pi}_{4}$ such that a generator matrix for $S\in {\mathcal G}_{3}(F_3^6)$ is
 $M=\left[\begin{matrix}1&0&0&0&0&0&0\\0&1&0&0&1&0&1\\
0&0&1&0&0&1&0\end{matrix}\right]$. By Lemma \ref{cardinality2} there are exactly four $4$-dimensional projective codes containing $S$.  A generator matrix for any such code
 can be obtained from $M$ by adding one of the vectors: \begin{align*}w_1=[0,0,0,1,1,1,2],\\ w_2=[0,0,0,1,1,2,2],\\ w_3=[0,0,0,2,1,1,2],\\ w_4=[0,0,0,2,1,2,2].\end{align*}
 From  $\dim\langle\{w_1, w_2, w_3, w_4\}\rangle=3$ and Lemma \ref{dimwi} it follows that $[S\rangle^{\Pi}_{4}$ is a star of $\Pi[6, 4]_3$.
\item
Choose $M=\left[\begin{matrix}v_1\\v_2\end{matrix}\right]=\left[\begin{matrix}1&0&0&1&1\\0&1&0&0&0\end{matrix}\right]$. As in the previous example, the number of equivalence classes of  $M$ is  equal $k$, but here only one class contains more than one element. The set $[S\rangle^{\Pi}_{3}$, where $S\in {\mathcal G}_{2}(F_3^5)$, consists of two codes whose generator matrices are of the form:
 $$\left[\begin{matrix}v_1\\v_2\\w_1\end{matrix}\right]=\left[\begin{matrix}1&0&0&1&1\\0&1&0&0&0\\0&0&1&1&2\end{matrix}\right],\ \ \ \left[\begin{matrix}v_1\\v_2\\w_2\end{matrix}\right]=\left[\begin{matrix}1&0&0&1&1\\0&1&0&0&0\\0&0&2&1&2\end{matrix}\right].$$
  The $3$-dimensional code with a generator matrix  $$\left[\begin{matrix}v_1\\v_2+w_1\\w_2\end{matrix}\right]=\left[\begin{matrix}1&0&0&1&1\\0&1&1&1&2\\0&0&2&1&2\end{matrix}\right]$$ is
       projective and its intersection with any element of $[S\rangle^{\Pi}_{3}$ is $2$-dimensional. Equivalently, this code along with codes from $[S\rangle^{\Pi}_{3}$ make up a clique of $\Pi(5, 3)_3$. So $[S\rangle^{\Pi}_{3}$ is not a maximal clique of $\Pi(5, 3)_3$.\end{enumerate}
\end{exmp}

\footnotesize Edyta Bartnicka\\
Institute of Information Technology,
Faculty of Applied Informatics
and Mathematics,
Warsaw University of Life Sciences - SGGW,
Nowoursynowska 166 St.,
02-787 Warsaw,
Poland\\
{\tt edyta\_bartnicka@sggw.edu.pl}
\normalsize
\end{document}